\documentclass[11pt,twoside]{article}
\usepackage{amsmath,amssymb,amsfonts,amsthm,graphicx}
\usepackage{hyperref,authblk,lineno,fancyhdr,lipsum}
\usepackage[titletoc,title]{appendix}

\newtheorem{theorem}{Theorem}[section]
\newtheorem{proposition}[theorem]{Proposition}
\newtheorem{corollary}[theorem]{Corollary}

\newtheorem{remark}[theorem]{Remark}
\newtheorem{definition}[theorem]{Definition}
\numberwithin{equation}{section}

\fancypagestyle{mypagestyle}{%
  \fancyhf{}
  \fancyhead[ER]{G. I. Sebe, D. Lascu}
  \fancyhead[OL]{Two asymptotic distributions related to R\'enyi-type continued fraction expansions}
  \fancyhead[RO]{\makebox[0pt][l]{{\thepage}}}
  \fancyhead[LE]{\makebox[0pt][r]{{\thepage}}}
}

\title{\textbf{Two asymptotic distributions related to R\'enyi-type continued fraction expansions}}

\author[1]{Gabriela Ileana Sebe \thanks{{igsebe@yahoo.com}}}
\affil[1]{Politehnica University of Bucharest, Faculty of Applied Sciences, Splaiul Independentei 313, 060042 Bucharest, Romania}
\affil[1]{Institute of Mathematical Statistics and Applied Mathematics, Calea 13 Sept. 13, 050711 Bucharest, Romania}

\author[2]{Dan Lascu \thanks{{lascudan@gmail.com}}}
\affil[2]{Mircea cel Batran Naval Academy, 1 Fulgerului, 900218 Constanta, Romania}

\begin{document}
\date{}
\maketitle

\noindent \textbf{Abstract}

\noindent
We attempt to investigate a two-dimensional Gauss-Kuzmin theorem for R\'enyi-type continued fraction expansions.
More precisely speaking, our focus is to obtain specific lower and upper bounds for the error term considered which imply the convergence rate of the distribution function involved to its limit.
To achieve our goal, we exploit the significant properties of the Perron-Frobenius operator of the R\'enyi-type map under its invariant measure on the Banach space of functions of bounded variation.
Finally, we give some numerical calculations to conclude the paper.
\\

\\


\sloppy

\section{Introduction}

The subject of R\'enyi-type continued fractions has link with $u$-backward continued fractions studied by Gr\"ochenig and Haas \cite{Grochenig&Haas-1996}.

As is known, in 1957 R\'enyi \cite{Renyi-1957} showed that every irrational number $x \in [0, 1)$ has an infinite continued fraction expansion of the form
\begin{equation}
x = 1 - \displaystyle \frac{1}{n_1 - \displaystyle \frac{1}{n_2 - \displaystyle \frac{1}{n_3 - \ddots }}} =:[n_1, n_2, n_3, \ldots]_b, \label{1.1}
\end{equation}
where each $n_i$ is an integer greater than $1$.
The expansion in (\ref{1.1}) is called the \textit{backward continued fraction expansion} of $x$.

The underlying dynamical system is the R\'enyi map $R$ defined from $[0, 1)$ to $[0, 1)$ by
\begin{equation}
R(x):=\frac{1}{1-x}- \left\lfloor\frac{1}{1-x}\right\rfloor \label{1.2}
\end{equation}
where $\lfloor \cdot \rfloor$ stands for the integer part.
R\'enyi showed that the infinite measure ${\mathrm{d}x}/{x}$ is invariant for $R$. This map does not possess a finite absolutely continuous invariant measure and the usual trick to investigate its thermodynamic formalism does not work.

Starting from the expansion (\ref{1.1}) and the R\'enyi transformation $R$, Gr\"ochenig and Haas \cite{Grochenig&Haas-1996} define the family of maps
$T_u (x) :=  \frac{1}{u(1-x)} - \lfloor \frac{1}{u(1-x)} \rfloor$, where $u>0$, $x \in [0, 1)$.
Given $u \in (0, 4)$ and $x \in [0, 1)$, $x$ has the $u$-backward continued fraction expansion
\begin{equation}
x = 1 - \displaystyle \frac{1}{u n_1 - \displaystyle \frac{1}{n_2 - \displaystyle \frac{1}{u n_3 - \displaystyle \frac{1}{n_4 - \ddots} }}} =:[a_1, a_2, a_3, \ldots]_u, \label{1.3}
\end{equation}
where the integers $n_i=1+a_i$ are $\geq 2$ and the coefficient of $n_i$ is $1$ or $u$, depending on the parity of $i$.
In the particular case $u=1/N$, for a positive integer $N \geq 2$, they have identified a finite absolutely continuous invariant measure for $T_u$, namely $\mathrm{d}x/(x+N-1)$.
For $u=1/N$, where $N \geq 2$ is an integer, we will call $T_u$ the \textit{Re\'nyi-type continued fraction transformation} and denote it by $R_N$.

The present paper continues and completes our series of papers dedicated to Re\'nyi-type continued fraction expansions \cite{LS-2020-1,LS-2020-2,SL-2020-3}.

In \cite{LS-2020-1} we started an approach to the metrical theory of the Re\'nyi-type continued fraction expansions via dependence with complete connections. We obtained a version of the Gauss-Kuzmin theorem for these expansions by applying the theory of random systems with complete connections, due to Iosifescu \cite{IG-2009}. Briefly, we showed that the associated random systems with complete connections are with contraction and their transition operators are regular with respect to the Banach space of Lipschitz functions.

In \cite{SL-2020-3} using a Wirsing-type approach \cite{Wirsing} we obtained upper and lower bounds of the error which provide a refined estimate of the convergence rate.
For example, in case $N=100$, the upper and lower bounds of the convergence rate are respectively $\mathcal{O}\left(w^n_{100}\right)$ and $\mathcal{O}\left(v^n_{100}\right)$ as $n \rightarrow \infty$, with $v_{100} > 0.00503350150708559$ and $w_{100} < 0.00503358526129032$.
The strategy in this paper was to restrict the domain of the Perron-Frobenius operator of $R_N$ under its invariant measure $\rho_N$ to the Banach space of functions which have a continuous derivative on $(0, 1)$.

Recently, in \cite{LS-2020-2} using the method of Sz\"usz \cite{Szusz-1961}, we obtained more information on the convergence rate involved. The main novelty was the explicit expression in terms of Hurwitz zeta functions on $\eta_N$ that appears in \cite[Theorem 3.1]{LS-2020-2}.
Finally, to enable direct comparisons of the results obtained in the last two methods (Wirsing and Sz\"usz) we give upper and lower bounds of $\eta_N$ for $N=100$: $0.00505050495049505<\eta_N<0.0050753806723955975$.

The aim of this paper is to contribute a solution to two-dimensional Gauss-Kuzmin theorem for Re\'nyi-type continued fraction expansions.

The framework of this paper is arranged as follows.
In Section 2 we gather prerequisites needed to prove our results in Section 3 and 4.
In Section 3 we treat the Perron-Frobenius operator of $R_N$ under its invariant measure on the Banach space of functions of bounded variation and study the significant properties of this operator.
Section 4 is devoted to the two-dimensional Gauss-Kuzmin theorem concerning the natural extension of corresponding interval maps $R_N$, $N \geq 2$.
Here the specific lower and upper bounds for the error term considered are approached via the characteristic properties of the associated transfer operator in Section 3. Finally, we give some remarks and numerical calculations to conclude the paper.

\section{Prerequisites}

In this section we briefly present known results about R\'enyi-type continued fractions (see e.g. \cite{LS-2020-1}).

\subsection{R\'enyi-type continued fraction expansions as dynamical system}

For a fixed integer $N \geq 2$, we define the R\'enyi-type continued fraction transformation $R_N: [0,1] \rightarrow [0,1]$ by
\begin{equation}
R_{N}(x) :=
\left\{
\begin{array}{ll}
{\displaystyle \frac{N}{1-x}- \left\lfloor\frac{N}{1-x}\right\rfloor},&
{ x \in [0, 1) }\\
0,& x=1.
\end{array}
\right. \label{2.1}
\end{equation}
%


For any irrational $x \in [0, 1]$, $R_N$ generates a new continued fraction expansion of $x$ of the form
\begin{equation} \label{2.2}
x = 1 - \displaystyle \frac{N}{1+a_1 - \displaystyle \frac{N}{1+a_2 - \displaystyle \frac{N}{1+a_3 - \ddots}}} =:[a_1, a_2, a_3, \ldots]_R,
\end{equation}
where $a_n$'s are positive integers greater than or equal to $N$ defined by
\begin{equation} \label{2.3}
a_1:=a_1(x) = \left\lfloor \frac{N}{1-x} \right\rfloor, x \neq 1; \quad a_1(1)=\infty
\end{equation}
and
\begin{equation}
a_n := a_n(x) = a_1\left( R^{n-1}_N (x) \right), \quad n \geq 2, \label{2.4}
\end{equation}
with $R_{N}^0 (x) = x$.

The R\'enyi-type continued fraction in (\ref{2.2}) can be viewed as a measure preserving dynamical system $\left([0,1],{\mathcal B}_{[0,1]}, R_N, \rho_N \right)$,
where $\mathcal{B}_{[0,1]}$ denotes the $\sigma$-algebra of all Borel subsets of $[0,1]$,
and
\begin{equation}
\rho_N (A) :=
\frac{1}{\log \left(\frac{N}{N-1}\right)} \int_{A} \frac{\mathrm{d}x}{x+N-1}, \quad A \in {\mathcal{B}}_{[0,1]} \label{2.5}
\end{equation}
is the invariant probability measure under $R_N$ \cite {Grochenig&Haas-1996}.

The R\'enyi-type continued fraction expansion (\ref{2.2}) is convergent.
To this end, define real functions $p_n(x)$ and $q_n(x)$, for $n \in \mathbb{N}:=\{0, 1, 2, \ldots\}$, by
\begin{eqnarray}
p_0&=&1, \ p_1=1+a_1-N, \quad p_n = (1+a_n) p_{n-1} - N p_{n-2}, n \geq 2, \label{2.6} \\
q_0&=&1, \ q_1=1+a_1,  \ \ \ \ \quad \quad q_n = (1+a_n) q_{n-1} - N q_{n-2}, n \geq 2. \label{2.7}
\end{eqnarray}
It follows that $p_n(x) / q_n(x) = [a_1, a_2, \ldots, a_n ]_R$ which is called the \textit{$n$-th order convergent} of $x \in [0, 1]$.
A simple inductive argument gives
\begin{equation}
p_{n-1}q_n - p_nq_{n-1} = N^n, \quad n \in \mathbb{N_+}:=\{1, 2, \ldots\} \label{2.8}
\end{equation}
and we obtain
\begin{equation}
\left| x - \frac{p_n}{q_n} \right| \leq \frac{N^n}{q_n(q_n - q_{n-1})}, \quad n \in \mathbb{N}_+. \label{2.9}
\end{equation}

Put $\Lambda:=\{N, N+1, \ldots\}$.
An $n$-block $(a_1, a_2, \ldots, a_n)$ is said to be \textit{admissible} for the expansion in (\ref{2.2}) if there exists $x \in [0, 1)$ such that $a_i(x)=a_i$ for all $1 \leq i \leq n$.
If $(a_1, a_2, \ldots, a_n)$ is an admissible sequence, we call the set
\begin{equation}
I (a_1, a_2, \ldots, a_n) = \{x \in [0, 1]:  a_1(x) = a_1, a_2(x) = a_2, \ldots, a_n(x) = a_n \} \label{2.10}
\end{equation}
\textit{the $n$-th order cylinder}. As we mentioned above, $(a_1, a_2, \ldots, a_n) \in \Lambda^n$.
For example, for any $a_1=i \in \Lambda$ we have
\begin{equation}
I\left( a_1\right) = \left\{x \in [0, 1]: a_1(x) = a_1 \right\} = \left[ 1 - \frac{N}{i}, 1 - \frac{N}{i+1} \right). \label{2.11}
\end{equation}

\subsection{Natural extension of $R_N$}

Let $\left([0, 1],{\mathcal B}_{[0, 1]}, R_N \right)$ be as in Section 2.1.
Define $\left(u^{i}_{N}\right)_{i \geq N}$ by
\begin{equation}
u^{i}_{N}: [0, 1] \rightarrow [0, 1]; \quad
u^{i}_{N}(x) := 1 - \frac{N}{x+i}, \quad x \in [0, 1]. \label{2.12}
\end{equation}
\noindent
For each $i \geq N$, $u^{i}_{N}$ is a right inverse of $R_N$, that is,
\begin{equation} \label{2.13}
\left(R_N \circ u^{i}_{N}\right)(x) = x, \quad \mbox{for any } x \in [0, 1].
\end{equation}
Furthermore, if $a_1(x)=i$, then $\left(u^{i}_{N} \circ R_N \right)(x)=x$ where $a_1$ is as in (\ref{2.3}).
\begin{definition} \label{def.natext}
The natural extension $\left([0, 1]^2, {\mathcal B}_{[0, 1]^2},\overline{R}_N \right)$ of $\left([0, 1],{\mathcal B}_{[0, 1]}, R_N \right)$ is the transformation $\overline{R}_N$ of the square space
$\left([0, 1]^2,{\mathcal B}_{[0, 1]}^2 \right):=\left([0, 1], {\mathcal B}_{[0, 1]}\right) \times \left([0, 1], {\mathcal B}_{[0, 1]}\right)$
defined as follows \cite{Nakada-1981}:
\begin{eqnarray} \label{2.14}
&&\overline{R}_N: [0, 1]^2 \rightarrow [0, 1]^2; \nonumber \\
&&\overline{R}_N(x,y) := \left( R_N(x), \,u^{a_1(x)}_{N}(y) \right), \quad (x, y) \in [0, 1]^2.
\end{eqnarray}
\end{definition}

From (\ref{2.13}), we see that $\overline{R}_N$ is bijective on $[0, 1]^2$ with the inverse
\begin{equation} \label{2.15}
(\overline{R}_N)^{-1}(x, y)
= (u^{a_1(y)}_{N}(x), \,
R_N(y)), \quad (x, y) \in [0, 1]^2.
\end{equation}

For $\rho_N$ in (\ref{2.5}), we define its \textit{extended measure} $\overline{\rho}_N$ on $\left([0, 1]^2, {\mathcal{B}}^2_{[0, 1]}\right)$ as
\begin{equation} \label{2.16}
\overline{\rho}_N(B) :=\frac{1}{ \log \left(\frac{N}{N-1}\right) } \int\!\!\!\int_{B}
\frac{N\mathrm{d}x\mathrm{d}y}{\left( N-(1-x)(1-y) \right)^2}, \quad B \in {\mathcal{B}}^2_{[0, 1]}.
\end{equation}
Then
$\overline{\rho}_N(A \times [0, 1]) = \overline{\rho}_N([0, 1] \times A) = \rho_N(A)$ for any $A \in {\mathcal{B}}_{[0, 1]}$.
The measure $\overline{\rho}_N$ is preserved by $\overline{R}_N$, i.e.,
$\overline{\rho}_N ((\overline{R}_N)^{-1}(B))
= \overline{\rho}_N (B)$ for any $B \in {\mathcal{B}}^2_{[0, 1]}$.

\subsection{Extended random variables}

Define the projection $E:[0, 1]^2 \rightarrow [0, 1]$ by $E(x,y):=x$.
With respect to $\overline{R}_N$ in (\ref{2.14}),
define \textit{extended incomplete quotients} $\overline{a}_l(x,y)$,
$l \in \mathbb{Z}:=\{\ldots, -2, -1, 0, 1, 2, \ldots\}$ at $(x, y) \in [0, 1]^2$ by
\begin{equation} \label{2.17}
\overline{a}_{l}(x, y) := (a_1 \circ E)\left(\,(\overline{R}_N)^{l-1} (x, y) \,\right),
\quad l \in \mathbb{Z}.
\end{equation}
Remark that $\overline{a}_{l}(x, y)$ in $(\ref{2.17})$ is also well-defined for $l \leq 0$ because $\overline{R}_N$ is invertible.
By $(\ref{2.12})$ and $(\ref{2.15})$ we have
\begin{equation} \label{2.18}
\overline{a}_n(x, y) = a_n(x), \quad
\overline{a}_0(x, y) = a_1(y), \quad
\overline{a}_{-n}(x, y) = a_{n+1}(y),
\end{equation}
for any $n \in \mathbb{N}_+$ and $(x, y) \in [0, 1]^2$.

Since the measure $\overline{\rho}_N$ is preserved by $\overline{R}_N$, the doubly infinite sequence
$(\overline{a}_l(x,y))_{l \in \mathbb{Z}}$
is strictly stationary (i.e., its distribution is invariant under a shift of the indices) under $\overline{\rho}_N$.
The stochastic property of $(\overline{a}_l(x,y))_{l \in \mathbb{Z}}$ follows from the fact that
\begin{equation} \label{2.19}
\overline{\rho}_N \left( [0, x] \times [0, 1] \,|
\,\overline{a}_0, \overline{a}_{-1}, \ldots \right)
= \frac{Nx}{N - (1-x)(1-a)} \quad \overline{\rho}_N \mbox{-}\mathrm{a.s.},
\end{equation}
for any $x \in [0, 1]$, where $a:= [\overline{a}_0, \overline{a}_{-1}, \ldots]_R$ with $\overline{a}_{l}:=\overline{a}_l(x,y)$ and $(x,y) \in [0, 1]^2$.

If $I_{n}$ denote the cylinder $I(\overline{a}_0, \overline{a}_{-1}, \ldots, \overline{a}_{-n})$ for $n \in \mathbb{N}$ for $l \in {\mathbb Z}$
and
$\left( \overline{a}_{1} = i \right) = I(i) \times [0, 1]$, $i \in \Lambda$,
it follows that
\begin{equation} \label{2.20}
\overline{\rho}_N (\left.\overline{a}_1 = i\right| \overline{a}_0, \overline{a}_{-1}, \ldots) = P^{i}_{N}(a) \quad \overline{\rho}_N \mbox{-}\mathrm{a.s.}
\end{equation}
where $a = [\overline{a}_0, \overline{a}_{-1}, \ldots]_R$ and
\begin{equation}
P^{i}_{N}(x) := \frac{x+N-1}{(x+i)\,(x+i-1)}. \label{2.21}
\end{equation}
The strict stationarity of $\left(\overline{a}_l\right)_{l \in \mathbb{Z}}$, under $\overline{\rho}_N$ implies that
\begin{equation} \label{2.22}
\overline{\rho}_N(\left.\overline{a}_{l+1} = i\, \right|\, \overline{a}_l,
\overline{a}_{l-1}, \ldots)
= P^{i}_{N}(a) \quad \overline{\rho}_N \mbox{-}\mathrm{a.s.}
\end{equation}
for any $i \geq N$ and $l \in \mathbb{Z}$, where
$a = [\overline{a}_l, \overline{a}_{l-1}, \ldots]_R$.

Motivated by (\ref{2.19}), we shall consider the one-parameter family $\{\rho^{t}_{N}: t \in [0, 1]\}$
of (conditional) probability measures on $\left([0, 1], {\mathcal{B}}_{[0, 1]} \right)$
defined by their distribution functions
\begin{equation}
\rho^{t}_{N} ([0, x]) := \frac{Nx}{N - (1-x)(1-t)}, \quad x, t \in [0, 1]. \label{2.23}
\end{equation}
Note that $\rho^{1}_{N} = \lambda$.

Let $a_{n}$'s be as in (\ref{2.4}).
For any $t \in [0, 1]$ put
\begin{equation}
s^{t}_{N,0} := t,\quad
s^{t}_{N,n} := 1 - \frac{N}{a_n + s^{t}_{N,n-1}}, \quad n \in \mathbb{N}_+. \label{2.24}
\end{equation}

Note that by the very definition of $s^{t}_{N,n}$, we have
\begin{equation}
s^{t}_{N,n} = [a_n, \ldots, a_2, a_1+t-1]_R, \, n \geq 2,
\end{equation}
while $s^{t}_{N,1} = 1- N/(a_1 + t)$,
$t \in [0, 1]$.
These facts lead us to the random system with complete connections \cite{LS-2020-1}
$\{([0, 1], {\mathcal{B}}_{[0, 1]}), \Lambda, u, P\}$,
where $u:[0, 1] \times \Lambda \rightarrow [0, 1]$ is defined as
\begin{equation} \label{2.25}
u(x,i):=u_i(x)=u^{i}_{N}(x)
\end{equation}
with $u^{i}_{N}$ as in (\ref{2.12})
and
$P:[0, 1] \times \Lambda \rightarrow [0, 1]$ is defined as
\begin{equation} \label{2.26}
P(x,i):= P_i(x)= P^{i}_{N}(x)
\end{equation}
with $P^{i}_{N}$ as in (\ref{2.21}), for all $x \in [0, 1]$ and $i \in \Lambda$.

Then $\left(s^{t}_{N,n}\right)_{n \in \mathbb{N}_+}$ is an $[0, 1]$-valued Markov chain on $([0, 1], {\mathcal{B}}_{[0, 1]}, \rho^{t}_{N})$ which starts from $s_{N,0}^{t} = t$, $t \in [0, 1]$, and has the following transition mechanism:
from state $s \in [0, 1]$ the only possible transitions are those to states $1-N/(s+i)$ with the corresponding transition probability $P^{i}_{N}(s)$, $i \in \Lambda$.

Let $B([0, 1])$ denote the Banach space of all bounded $[0, 1]$-measurable complex-valued functions defined on $[0, 1]$ which is a Banach space under the supremum norm.
The transition operator of $(s^{t}_{N,n})_{n \in \mathbb{N}_+}$ takes $f \in B([0, 1])$ into the function defined by
\begin{equation} \label{2.27}
E_{\rho_N^t}\left( \left. f(s^{t}_{N,n+1})\right| s^{t}_{N,n} = s \right) = \sum_{i \in \Lambda}P^{i}_{N}(s)f\left(u^{i}_N\right) = U_N f(s)
\end{equation}
for any $s \in I$, where $E_{\rho_N^t}$ stands for the mean-value operator with respect to the probability measure $\rho^{t}_{N}$, whatever $t \in [0, 1]$, and $U_N$ is the Perron-Frobenius operator of $([0, 1],{\mathcal B}_{[0, 1]}, \rho_N, R_N)$ defined as in (\ref{3.1}).

Note that for any $t \in [0, 1]$ and $n \in \mathbb{N}_+$ we have
\begin{equation} \nonumber
\rho^{t}_{N}\left(A| a_1, \ldots, a_n\right) = \rho^{s^{t}_{N,n}}_{N}\left(R^{n}_N(A)\right),
\end{equation}
whatever the set $A$ belonging to the $\sigma$-algebra generated by the random variables $a_{n+1}, a_{n+2} \ldots$, that is,
$\sigma(a_{n+1}, a_{n+2}, \ldots) = R^{-n}_N \left(\mathcal{B}_{[0, 1]} \right)$.
In particular, it follows that the Brod\'en-Borel-L\'evy formula holds under $\rho^{t}_{N}$ for any $t \in [0, 1]$, that is,
\begin{equation}
\rho^{t}_{N} (R^n_N < x |a_1,\ldots, a_n)
= \frac{Nx}{N-(1-x)(1-s^{t}_{N,n})}, \quad x \in [0, 1], n \in \mathbb{N}_+. \label{2.28}
\end{equation}

\section{Perron-Frobenius operator of $R_N$}

We shall discuss the relevant properties of the Perron-Frobenius operator of $R_N$ under the invariant measure $\rho_N$ and related problems in terms of specified operator domain.

Let $([0, 1],{\mathcal B}_{[0, 1]},\rho_{N}, R_{N})$ be as in (\ref{2.1}) and (\ref{2.5}) and let
$L^1([0, 1],\rho_{N}):=\{f: [0, 1] \rightarrow \mathbb{C} : \int^{1}_{0} |f |\mathrm{d}\rho_{N} < \infty \}$.
The \textit{Perron-Frobenius operator} of $([0, 1],{\mathcal B}_{[0, 1]},\rho_{N}, R_{N})$ is defined as the bounded linear operator $U_N$ on the Banach space $L^1([0, 1],\rho_{N})$ such that the following holds \cite{LS-2020-1}:
\begin{equation}
U_Nf(x) = \sum_{i \geq N}P^{i}_{N}(x)\,f\left(u^{i}_N(x)\right), \quad f \in L^1([0, 1],\rho_{N}) \label{3.1}
\end{equation}
where $P_N^{i}$ and $u^{i}_N$ are as in (\ref{2.21}) and (\ref{2.12}), respectively.

For a function $f: [0, 1] \rightarrow {\mathbb C}$, define the \textit{variation} $\mathrm{var}_{A}f$ of $f$ on a subset $A$ of $[0, 1]$ by
\begin{equation}
{\rm var}_A f := \sup \sum^{k-1}_{i=1} |f(y_{i+1}) - f(y_{i})|, \label{3.2}
\end{equation}
where the supremum being taken over $y_1 < \cdots < y_k$, $y_i \in A$, $i =1, \ldots, k$ and $k \geq 2$.
We write simply $\mathrm{var} f$ for $\mathrm{var}_{[0, 1]} f$.
Let $BV([0, 1]):=\{f:[0, 1] \rightarrow {\mathbb C}: {\rm var}\,f<\infty\}$
under the norm
$
\left\| f \right\|_\mathrm{v} := \mathrm{var }f + |f|,
$
where $|f| := \sup_{x \in [0, 1]} |f(x)|$.

Next, we calculate the variation of the Perron-Frobenius operator.
\begin{proposition} \label{prop3.1}
For any $f \in BV([0, 1])$ we have
\begin{equation} \nonumber
\mathrm{var}\,U_Nf \leq \frac{1}{N} \cdot \mathrm{var } f + K_N \cdot |f|,
\end{equation}
where
\begin{equation}\label{3.3}
K_N:= \frac{2}{2N-1+2\sqrt{N(N-1)}}.
\end{equation}
\end{proposition}
\begin{proof}
Recall that $P^{i}_N(x) = \displaystyle \frac{i+1-N}{x+i} - \displaystyle \frac{i-N}{x+i-1}$, $i \geq N$.
We have
\begin{equation*}
\left(P^{i}_N(x)\right)' = \frac{i-N}{(x+i-1)^2} -  \frac{i+1-N}{(x+i)^2} = \frac{L(N,x)}{(x+i-1)^2(x+i)^2}
\end{equation*}
with $L(N,x)=-x^2+2x(1-N)+i^2+i(1-2N)+N-1$, for every $i \geq N$.

If $N \leq i \leq 2N-2$, then $L(N,x)<0$ for all $x \in [0, 1]$, i.e., $\left(P^{i}_N(x)\right)'<0$, $x \in [0, 1]$.
Hence, the functions $P^i_N$ are decreasing on $[0, 1]$.

If $i=2N-1$, then $L(N,x)>0$ for all $x \in \left[0, 1-N+\sqrt{N(N-1)}\right]$,
and $L(N,x)<0$ for all $x \in \left(1-N+\sqrt{N(N-1)}, 1\right]$.
Hence $P^{2N-1}_N$ is increasing on $\left[0, 1-N+\sqrt{N(N-1)}\right]$ and decreasing on $\left(1-N+\sqrt{N(N-1)}, 1\right]$.

If $i \geq 2N$, then $L(N,x)>0$ for all $x \in [0, 1]$, i.e., $\left(P^{i}_N(x)\right)'>0$, $x \in [0, 1]$.
Hence, the functions $P^{i}_N$ are increasing on $[0, 1]$.

Hence
\begin{eqnarray*}
  \mathrm{var} \, P^{i}_N = \left\{\begin{array}{lll}
   P^{i}_N(0) - P^{i}_N(1), && N \leq i \leq 2N-2 \\
   2 P^{i}_N(1-N+\sqrt{N(N-1)}) - P^{i}_N(0) - P^{i}_N(1), && i=2N-1 \\
   P^{i}_N(1) - P^{i}_N(0), && i \geq 2N.
\end{array} \right.
\end{eqnarray*}
and
\begin{eqnarray*}
|P^{i}_N| = \sup_{x \in [0, 1]} P^{i}_N(x) = \left\{\begin{array}{lll}
                              P^{i}_N(0),  && N \leq i \leq 2N-2 \\
                              P^{i}_N (1-N+\sqrt{N(N-1)}), && i=2N-1 \\
                              P^{i}_N(1), && i \geq 2N.
\end{array} \right.
\end{eqnarray*}
Thus
\begin{eqnarray*}
\sup_{i \geq N} |P^{i}_N| &=& \max \left\{P^{N}_N(0), P^{2N-1}_N(1-N+\sqrt{N(N-1)}), P^{2N}_N(1) \right\} \\
                          &=& \max \left\{\frac{1}{N}, \frac{1}{(\sqrt{N}+\sqrt{N-1})^2}, \frac{1}{2(1+2N)}\right\} = \frac{1}{N}.
\end{eqnarray*}
Also,
\begin{eqnarray*}
\sum_{i \geq N} \mathrm{var} \, P^{i}_N &=& \sum_{N \leq i \leq 2N-2} \left( P^{i}_N(0) - P^{i}_N(1)\right) +
                                          \mathrm{var } \, P^{2N-1}_N + \sum_{i \geq 2N} \left( P^{i}_N(1) - P^{i}_N(0)\right) \\
&=& \frac{1}{2(2N-1)} + \frac{2}{2N-1+2\sqrt{N(N-1)}}-\frac{1}{2N-1}+\frac{1}{2(2N-1)}\\
&=& \frac{2}{2N-1+2\sqrt{N(N-1)}}.
\end{eqnarray*}
Therefore
\begin{eqnarray*}
\mathrm{var} \, U_Nf &=& \mathrm{var} \sum_{i \geq N} P^{i}_N \cdot (f \circ u^i_N) \leq \sum_{i \geq N} \mathrm{var} \left( P^{i}_N \cdot (f \circ u_N^i)\right) \\
&\leq& \sum_{i \geq N} |P^{i}_N| \mathrm{var} (f \circ u^i_N) +  \sum_{i \geq N} |f \circ u^i_N| \mathrm{var} \, P^{i}_N \\
&\leq& \left( \sup_{i \geq N} |P^{i}_N| \right)\sum_{i \geq N} \mathrm{var} (f \circ u^i_N) + |f| \sum_{i \geq N} \mathrm{var} \, P^{i}_N \leq \frac{1}{N} \cdot \mathrm{var } f + K_N \cdot |f|,
\end{eqnarray*}
where the constant $K_N$ is as in (\ref{3.3}) and because we took into account that
\begin{equation} \nonumber
\sum_{i \geq N} \mathrm{var} (f \circ u^i_N) = \sum_{i \geq N} \mathrm{var}_ {\left[1-\frac{N}{i}, 1-\frac{N}{i+1} \right]} \, f = \mathrm{var} \, f.
\end{equation}
\end{proof}

If $f \in B([0, 1])$, define the linear functional $U_N^{\infty}$ by
\begin{equation}
U_N^{\infty}:B([0, 1])\to {\mathbb C}; \quad U_N^{\infty} f = \int^{1}_{0}f(x)\,\rho_{N}(\mathrm{d} x). \label{3.4}
\end{equation}
Then we have
\begin{equation} \label{3.5}
U_N^{\infty} U_N^n f = U_N^{\infty} f \quad \mbox{for any } n \in \mathbb{N}_+.
\end{equation}

\begin{corollary} \label{cor.3.2}
For any $f \in BV(I)$ and for all $n \in \mathbb{N}$ we have
\begin{eqnarray}
 \mathrm{var}\,U_N^n f &\leq& \left( \frac{1}{N} + K_N \right)^n \cdot \mathrm{var } f, \label{3.6} \\
  \left| U_N^n f - U_N^{\infty} f \right| &\leq& \left( \frac{1}{N} + K_N \right)^n \cdot \mathrm{var } f. \label{3.7}
\end{eqnarray}
\end{corollary}
\begin{proof}
Note that for any $f \in BV([0, 1])$ and $u \in [0, 1]$, since $ \int^{1}_{0} \mathrm{d}\rho_N(x) = 1$, we have
\begin{eqnarray*}
|f(u)| -  \left| \int^{1}_{0} f(x) \mathrm{d}\rho_N (x) \right| &\leq& \left| f(u)- \int^{1}_{0} f(x) \mathrm{d}\rho_N(x) \right| \\
&=&  \left| \int^{1}_{0}(f(u) - f(x)) \mathrm{d}\rho_N(x)\right| \leq \mathrm{var } f,
\end{eqnarray*}
whence
\begin{equation}\label{3.8}
|f| = \sup_{u \in [0, 1]} |f(u)| \leq  \left| \int^{1}_{0} f(x) \mathrm{d}\rho_N(x) \right| + \mathrm{var } f, \quad f \in BV([0, 1]).
\end{equation}
Finally, (\ref{3.4}), (\ref{3.5}) and (\ref{3.8}) imply that
\begin{eqnarray}
\left| U_N^n f - U_N^{\infty} f \right| &\leq& \left| \int^{1}_{0} \left( U_N^n f - U_N^{\infty}f \right)(x) \mathrm{d}\rho_N(x) \right| +
\mathrm{var } \left( U_N^n f - U_N^{\infty} f \right) \nonumber \\
&\leq& \left| U_N^{\infty}U_N^n f - U_N^{\infty} f\right| + \mathrm{var } \, U_N^n f = \mathrm{var } \, U_N^n f, \label{3.9}
\end{eqnarray}
for all $n \in \mathbb{N}$ and $f \in BV([0, 1])$.

It follows from Proposition \ref{prop3.1} that for all $f \in BV([0, 1])$ we have
\begin{equation*}
\mathrm{var } \left( U_N f - U_N^{\infty} f \right) \leq \frac{1}{N} \cdot \mathrm{var } \left( f - U_N^{\infty} f \right)
+ K_N \cdot \left| f - U_N^{\infty} f \right|.
\end{equation*}
But,
\[
\mathrm{var } \left( U_N f - U_N^{\infty} f \right) = \mathrm{var } \, U_N f, \quad \mathrm{var } \left( f - U_N^{\infty} f \right) = \mathrm{var } f,
\]
and $\left| f - U_N^{\infty} f \right| \leq \mathrm{var } f$  which is (\ref{3.9}) with $n=0$.
Thus,
\begin{equation*}
\mathrm{var } \, U_N f \leq \frac{1}{N} \cdot \mathrm{var } f + K_N \cdot \mathrm{var } f = \left( \frac{1}{N} + K_N \right) \cdot \mathrm{var } f
\end{equation*}
which leads to (\ref{3.6}). Next, (\ref{3.7}) follows from (\ref{3.9}) and (\ref{3.6}).
\end{proof}

By induction with respect to $n \in \mathbb{N}_+$ we get
\begin{equation} \label{3.10}
U_N^n f(x) = \sum_{i_1, \ldots, i_n \in \Lambda} P_N^{i_1 \ldots i_n} (x) f(u_N^{i_n \ldots i_1}(x)), \quad x \in [0, 1]
\end{equation}
where
\begin{eqnarray}
  u_N^{i_n \ldots i_1}   &=& u_N^{i_n} \circ \ldots \circ u_N^{i_1} \label{3.11}\\
  P_N^{i_1 \ldots i_n} (x) &=& P_N^{i_1} (x) P_N^{i_2} (u_N^{i_1} (x)) \ldots P_N^{i_n} (u_N^{i_{n-1} \ldots i_1} (x)), \quad n \geq 2, \label{3.12}
\end{eqnarray}
and the functions $u_N^{i}$ and $P_N^i$ are defined in (\ref{2.12}) and (\ref{2.21}), respectively, for all $i \in \Lambda$.

Putting
\begin{equation} \nonumber
\frac{p_n(i_1, \ldots, i_n)}{q_n(i_1, \ldots, i_n)} = [i_1, \ldots, i_n]_R, \quad n \in \mathbb{N}_+,
\end{equation}
for arbitrary indeterminates $i_1, \ldots, i_n$, we get
\begin{eqnarray} \label{3.14}
P_N^{i_1\ldots i_n} (t) &=& \frac{(t+N-1)N^{n-1}}{(t+i_1)q_{n-1}(i_2, \ldots, i_n) -N q_{n-2}(i_3, \ldots, i_{n-1}, i_n)}  \\
                      &\times& \frac{1}{(t+i_1)q_{n-1}(i_2, \ldots, i_{n-1}, i_n-1) -N q_{n-2}(i_3, \ldots, i_{n-1}, i_n-1)}\nonumber
\end{eqnarray}
for all $i_n \in \Lambda$, $n \geq 2$, and $t \in [0, 1]$.

\section{A two-dimensional Gauss–Kuzmin theorem}

In this section we shall deliver an estimate of the error term below
\begin{equation*}
e^t_{N,n} (x,y) = \rho^t_{N} \left( R_N^n \in [0,x], s^t_{N,n} \in [0,y] \right)  -
\frac{1}{\log \left( \frac{N}{N-1} \right)} \log \frac{(x+N-1)(y+N-1)}{(N-1)\left(N-(1-x)(1-y)\right)}
\end{equation*}
for any $t \in [0, 1]$, $x, y \in [0, 1]$ and $n \in \mathbb{N}_+$.

In the main result of this section, Theorem \ref{th.4.4}, we shall derive lower and upper bounds (not depending on $t \in [0, 1]$) of the supremum
\begin{equation}\label{4.1}
\sup_{x, y \in [0, 1]} |e^t_{N,n} (x, y)|, \quad t \in [0, 1],
\end{equation}
which provide an estimate of the convergence rate involved.
First, we obtain a lower bound for the error, which suggests the convergence rate of
$\rho^t_{N} \left( s^t_{N,n} \in [0,y] \right)$ to $\rho_{N} \left([0,y] \right)$ as $n \rightarrow \infty$ for all $t \in [0, 1]$.

\begin{theorem} \label{Th.4.1}
For any $t \in [0, 1]$ and $n \in \mathbb{N}_+$ we have
\begin{equation*}
\frac{1}{2} P_N^{N(n)}(1) \leq \sup_{y \in [0, 1]} \left|\rho^t_{N} \left( s^t_{N,n} \in [0,y] \right) - \rho_{N} \left([0,y] \right) \right|
\end{equation*}
with $P_N^{N(n)}(t) = \displaystyle \sup_{s \in [0, 1]} \rho^t_{N}\left( s^t_{N,n} = s \right)$, where we write $N(n)$ for $(i_1, \ldots, i_n)$ with $i_1=\ldots=i_n=N$,
$n \in \mathbb{N}_+$.
\end{theorem}

\begin{proof}
First, the continuity of the function $y \mapsto \rho_{N} \left([0,y] \right)$, $y \in [0, 1]$, and the equations
\[
  \lim_{h \searrow 0} \rho^t_{N} \left( s^t_{N,n} \leq y-h \right) = \rho^t_{N} \left( s^t_{N,n} < y \right)
\]
and
\[
  \lim_{h \searrow 0} \rho^t_{N} \left( s^t_{N,n} < y+h \right) = \rho^t_{N} \left( s^t_{N,n} \leq y \right)
\]
imply that
\begin{equation} \nonumber
\sup_{y \in [0,1]} \left|\rho^t_{N} \left( s^t_{N,n} \leq y \right) - \rho_{N} \left([0,y] \right) \right| =
\sup_{y \in [0,1]} \left|\rho^t_{N} \left( s^t_{N,n} < y \right) - \rho_{N} \left([0,y] \right) \right|
\end{equation}
for all $t \in [0, 1]$ and $n \in \mathbb{N}$.
Second, whatever $s \in [0, 1]$ we have
\begin{eqnarray*}
  \rho^t_{N} (s^t_{N,n}=s) &=& \rho^t_{N} \left( s^t_{N,n} \leq s \right)- \rho_{N} \left([0,s] \right) - \left(\rho^t_{N} \left( s^t_{N,n} < s \right) - \rho_{N} \left([0,s] \right) \right) \nonumber \\
  &\leq& \sup_{y \in [0, 1]} \left|\rho^t_{N} \left( s^t_{N,n} \leq y \right) - \rho_{N} \left([0,y] \right) \right| + \sup_{y \in [0, 1]} \left|\rho^t_{N} \left( s^t_{N,n} < y \right) - \rho_{N} \left([0,y] \right) \right| \nonumber \\
  &=& 2 \sup_{y \in [0, 1]} \left|\rho^t_{N} \left( s^t_{N,n} \leq y \right) - \rho_{N} \left( [0,y] \right) \right|.
\end{eqnarray*}
Hence
\begin{eqnarray*}
\sup_{y \in [0, 1]} \left|\rho^t_{N} \left( s^t_{N,n} \in [0,y] \right) - \rho_{N} \left( [0,y] \right) \right| &=&
\sup_{y \in [0, 1]} \left|\rho^t_{N} \left( s^t_{N,n} \leq y \right) - \rho_{N} \left( [0,y] \right) \right| \nonumber \\
&\geq& \frac{1}{2} \sup_{s \in [0, 1]} \rho^t_{N} \left( s^t_{N,n} = s \right),
\end{eqnarray*}
for all $t \in [0, 1]$ and $n \in \mathbb{N}$.
Next, using (\ref{2.27}) we have
\[
U_N^n f(t) = E_{\rho^t_N}\left( f\left(s^t_{N,n}\right)\right), \ n \in \mathbb{N}, f \in B([0, 1]), t \in [0, 1].
\]
As $s^t_{N,n} = u_N^{a_n, \ldots, a_1}(t)$, $t \in [0, 1]$, $n \in \mathbb{N}_+$, we have
\begin{equation} \label{3.13}
U_N^n f(t) = \sum_{i^{(n)} \in \Lambda^n} \rho^t_{N} \left( (a_1, a_2, \ldots, a_n) = i^{(n)} \right) f\left( u_N^{i_n \ldots i_1} (t)\right)
\end{equation}
for any $n \in \mathbb{N}_+$, $f \in B([0, 1])$, $t \in [0, 1]$ and $i^{(n)} = (i_1, \ldots, i_n) \in \Lambda^n$.
Hence, by (\ref{2.10}), (\ref{3.10}) and (\ref{3.13}) we get
\begin{equation} \nonumber
P_N^{i_1\ldots i_n}(t) = \rho^t_{N} \left( I_N \left( i^{(n)}\right) \right) = \rho^t_{N} \left( s^t_{N,n} = [i_n, \ldots, i_2, i_1+t-1]_R \right), \ n \geq 2,
\end{equation}
\begin{equation} \nonumber
P_N^{i_1}(t) = \rho^t_{N} \left( I_N \left( i_{1}\right) \right) = \rho^t_{N} \left( s^t_{N,1} = 1- \frac{N}{i_1 +t} \right),
\end{equation}
for all $t \in [0, 1]$ and $i_1,\ldots, i_n \in \Lambda$.

Since as easily seen,
\[
\max_{ i^{(n)} \in \Lambda^n } \rho^t_{N} \left( I_N \left( i^{(n)}\right) \right) = \rho^t_{N} \left( I_N \left( N(n)\right) \right),
\]
where we write $N(n)$ for $i^{(n)} = (i_1, \ldots, i_n)$ with $i_1 = \ldots = i_n=N$, $n \in \mathbb{N}_+$.

Also by (\ref{3.14}) we have
\begin{eqnarray*}
P_N^{N(n)} (t) &=& \frac{(t+N-1)N^{n-1}}{(t+N)q_{n-1}(\underbrace{N, \ldots, N}_{(n-1) \ times})-N q_{n-2}(\underbrace{N, \ldots, N}_{(n-2) \ times})} \nonumber \\
                      &\times& \frac{1}{(t+N)q_{n-1}(\underbrace{N, \ldots, N, N}_{n-2 \ times},N-1) - N q_{n-2}(\underbrace{N, \ldots, N}_{n-3 \ times}, N-1)}.
\end{eqnarray*}
It is easy to see that $P_N^{N(n)} (\cdot)$ is a decreasing function. Therefore
\begin{equation} \nonumber
\sup_{s \in [0, 1]} \rho^t_{N} \left( s^t_{N,n} = s \right) = P_N^{N(n)}(t) \geq P_N^{N(n)}(1)
\end{equation}
for all $t \in [0, 1]$.

\end{proof}

\begin{theorem} \label{th.4.2}
(The lower bound)
For any $t \in [0, 1]$ we have
\begin{equation*}
\begin{split}
\frac{1}{2} P_N^{N(n)}(1) \leq &\sup_{x,y \in [0, 1]} \left|\rho^t_{N} \left( R^n_N \in [0,x], s^t_{N,n} \in [0,y] \right) \right.\\
&\left. - \frac{1}{\log \left( \frac{N}{N-1} \right)} \log \frac{(x+N-1)(y+N-1)}{(N-1)\left(N-(1-x)(1-y)\right)} \right|
\end{split}
\end{equation*}
for all $n \in \mathbb{N}_+$.
\end{theorem}
\begin{proof}
Whatever $t \in [0, 1]$ and $n \in \mathbb{N}_+$, by Theorem \ref{Th.4.1} we have
\begin{eqnarray*}
&&\displaystyle \sup_{x,y \in [0, 1]} \left|\rho^t_{N} \left( R^n_N \in [0,x], s^t_{N,n} \in [0,y] \right) -
 \frac{1}{\log \left( \frac{N}{N-1} \right)} \log \frac{(x+N-1)(y+N-1)}{(N-1)\left(N-(1-x)(1-y)\right)} \right| \nonumber \\
&&\geq \sup_{y \in [0, 1]} \left|\rho^t_{N} \left( R^n_N \in [0, 1], s^t_{N,n} \in [0,y] \right) -
\frac{1}{\log\left(\frac{N}{N-1}\right)} \log\left( \frac{y+N-1}{N-1} \right) \right| \nonumber \\
&&= \sup_{y \in [0, 1]} \left|\rho^t_{N} \left( s^t_{N,n} \in [0,y] \right) - \rho_{N} \left( [0,y] \right) \right|
\geq \frac{1}{2} P_N^{N(n)}(1).
\end{eqnarray*}
\end{proof}

\begin{remark}
Since $q_n( \underbrace{N, \ldots, N}_{(n-1) \ times}, N-1 ) = N^n$ we get
\begin{equation} \nonumber
P_N^{N(n)}(1) = \frac{1}{q_n(N(n))}, \ n \in \mathbb{N}_+.
\end{equation}
By the recurrence relation (\ref{2.7}) with $a_n=i_n$ for all $n \in \mathbb{N}$, we obtain
\[
q_n(N(n))=\frac{N^{n+1}-1}{N-1}.
\]
It should be noted that Theorem \ref{th.4.2} in connection with the limit
\begin{equation} \nonumber
\lim_{n \rightarrow \infty} \left( \frac{1}{2} P_N^{N(n)}(1) \right)^{1/n} =
\lim_{n \rightarrow \infty} \left( \frac{N-1}{2\left(N^{n+1}-1\right)} \right)^{\frac{1}{n}} = \frac{1}{N}
\end{equation}
leads to an estimate of the order of magnitude of the error term $e^t_{N,n} (x,y)$.
\end{remark}

In what follows we exploit the characteristic properties of the transition operator associated with the random system with complete connections underlying R\'enyi-type continued fraction.
By restricting this operator to the Banach space of functions of bounded variation on $[0, 1]$, we derive an explicit upper bound for the supremum (\ref{4.1}).
\begin{theorem} \label{th.4.3}
(The upper bound)
For any $t \in [0, 1]$ we have
\begin{equation} \nonumber
\begin{split}
    \sup_{x, y \in [0, 1]} &\left|\rho^t_{N} \left( R^n_N \in [0,x], s^t_{N,n} \in [0,y] \right) -
\frac{1}{\log \left( \frac{N}{N-1} \right)} \log \frac{(x+N-1)(y+N-1)}{(N-1)\left(N-(1-x)(1-y)\right)} \right| \\
   & \leq \left(\frac{1}{N} + K_N \right)^n.
\end{split}
\end{equation}
for all $n \in \mathbb{N}$, where $K_N$ is as in (\ref{3.3}).
\end{theorem}
\begin{proof}
Let $F^t_{N,n}(y) = \rho^t_{N} (s^t_{N,n} \leq y)$ and $H^t_{N,n} (y) = F^t_{N,n}(y) - \rho_N ([0,y])$, $t, y \in [0, 1]$, $n \in \mathbb{N}$.
Note that $H^t_{N,n}(0) = 0$.
As we have noted $U_N$ is the transition operator of the Markov chain $(s^t_{N,n})_{n \in \mathbb{N}}$.
For any $y \in [0, 1]$ consider the function $f_y$ defined on $[0, 1]$ as
\begin{equation} \nonumber
f_{y}(t):=
\left\{
\begin{array}{ll}
{1} & { \mbox{if } \, 0 \leq t \leq y,}\\
{0} & \mbox{if } \, y < t \leq 1.
\end{array}
\right.
\end{equation}
Hence
\begin{equation} \nonumber
U_N^n f_y (t) = E_{\rho^t_N}\left( \left. f_y(s^t_{N,n})\right| s^t_{N,0} = t \right) = \rho^t_{N} (s^t_{N,n} \leq y)
\end{equation}
for all $t, y \in [0, 1]$, $n \in \mathbb{N}$.
As
\begin{equation} \nonumber
U_N^{\infty} f_y = \int^1_{0} f_y(t) \mathrm{d}\rho_N (t) = \rho_N ([0,y]), \quad y \in [0, 1].
\end{equation}
It follows from Corollary \ref{cor.3.2} that
\begin{eqnarray} \label{4.3}
  |H^t_{N,n}(y)| &=& \left| \rho^t_{N} (s^t_{N,n} \leq y) - \rho_N ([0,y]) \right| = \left| U_N^n f_y (t) - U_N^{\infty} f_y \right| \nonumber \\
                  &&\leq \left(\frac{1}{N} +K_N \right)^n \mathrm{var } \, f_y = \left(\frac{1}{N} + K_N \right)^n
\end{eqnarray}
for all $t, y \in [0, 1]$, $n \in \mathbb{N}$.
By the very definition of the conditional probability and (\ref{2.28}), for all $t \in [0, 1]$, $x, y \in [0, 1]$ and $n \in \mathbb{N}$ we have
\begin{equation*}
\begin{split}
    &\rho^t_{N} \left( R^n_N \in [0,x], s^t_{N,n} \in [0,y] \right) = \rho^t_{N} \left( \left.R^n_N \in [0,x] \, \right| \, s^t_{N,n} \in [0, y] \right) \cdot \rho^t_{N} (s^t_{N,N} \in [0,y]) \\
     & = \rho^t_{N} \left( \left.R^n_N \in [0,x] \, \right| \, s^t_{N,n} \in [0, y] \right) \cdot F^t_{N,n}(y)
     =  \int^{y}_{0} \rho^t_{N} \left( \left.R^n_N \in [0,x] \, \right| \, s^t_{N,n} = z \right) \mathrm{d}F^t_{N,n}(z) \\
   & = \int^{y}_{0} \frac{Nx}{N-(1-x)(1-z)} \mathrm{d}F^t_{N,n}(z) \\
   &= \int^{y}_{0} \frac{Nx}{N-(1-x)(1-z)} \mathrm{d}\rho_{N}(z) + \int^{y}_{0} \frac{Nx}{N-(1-x)(1-z)} \mathrm{d}H^t_{N,n}(z) \\
   & = \frac{1}{\log \left( \frac{N}{N-1} \right)} \log \frac{(x+N-1)(y+N-1)}{(N-1)\left(N-(1-x)(1-y)\right)} + \frac{Nx}{N-(1-x)(1-z)} \left.H^t_{N,n}(z)\right|^{y}_{0} \\
   & + \int^{y}_{0} \frac{Nx(1-x)}{(N-(1-x)(1-z))^2} H^t_{N,n}(z)\mathrm{dz}.
\end{split}
\end{equation*}
Hence, by (\ref{4.3})
\begin{eqnarray*}
&&\left|\rho^t_{N} \left( R^n_N \in [0,x], s^t_{N,n} \in [0,y] \right) -
\frac{1}{\log \left( \frac{N}{N-1} \right)}\log \frac{(x+N-1)(y+N-1)}{(N-1)\left(N-(1-x)(1-y)\right)} \right| \\
&& \leq \left(\frac{1}{N} + K_N \right)^n \left( \frac{Nx}{N-(1-x)(1-y)} - \left.\frac{Nx}{N-(1-x)(1-z)}\right|_{z=0}^{z=y} \right) \\
&& = \left(\frac{1}{N} + K_N \right)^n  \frac{Nx}{N-1+x} \leq \left(\frac{1}{N} + K_N \right)^n
\end{eqnarray*}
where $K_N$ is as in (\ref{3.3}), $t,x,y \in [0, 1]$, $n \in \mathbb{N}$.
\end{proof}

Combining Theorem \ref{th.4.2} with Theorem \ref{th.4.3} we obtain Theorem \ref{th.4.4}.
\begin{theorem} \label{th.4.4}
Whatever $t \in [0, 1]$ we have
\begin{equation*}
\begin{split}
 &\frac{1}{2} P_N^{N(n)}(1) \\
 &\leq \sup_{x,y \in [0, 1]} \left|\rho^t_{N} \left( R^n_N \in [0,x], s^t_{N,n} \in [0,y] \right) - \frac{1}{\log \left( \frac{N}{N-1} \right)}\log \frac{(x+N-1)(y+N-1)}{(N-1)\left(N-(1-x)(1-y)\right)} \right| \\
 & \leq \left(\frac{1}{N} + K_N \right)^n
\end{split}
\end{equation*}
for all $n \in \mathbb{N}_+$.
\end{theorem}

Actually, Theorem \ref{th.4.4} implies that the convergence rate is $\mathcal{O}(\alpha^n)$, with
\begin{equation*}
\frac{1}{N} \leq \alpha \leq \frac{1}{N} +\frac{2}{2N-1+2\sqrt{N(N-1)}}.
\end{equation*}

For example, we have
\\
\begin{center}
\begin{tabular}{|l|l|}
  \hline
  $N=2$ & $0.5 \leq \alpha \leq 0.843145\ldots$\\\hline
  $N=3$ & $0.33333\ldots \leq \alpha \leq 0.535374\ldots$\\\hline
  $N=5$ & $0.2 \leq \alpha \leq 0.311456\ldots$\\\hline
  $N=10$ & $0.1 \leq \alpha \leq 0.152668\ldots$\\\hline
  $N=100$ & $0.01 \leq \alpha \leq 0.0150252\ldots$\\\hline
  $N=1000$ & $0.001 \leq \alpha \leq 0.00150025\ldots$\\\hline
  $N=10000$ & $0.0001 \leq \alpha \leq 0.000150003\ldots$\\\hline
  \hline
\end{tabular}
\end{center}




%

\end{document}